\documentclass[a4paper,12pt,oneside]{amsart}
\usepackage{microtype}
\usepackage{amssymb,latexsym}
\usepackage[margin=3.25cm]{geometry}
\usepackage{mathtools}
\usepackage{amsthm}
\usepackage[inline]{enumitem}
\usepackage{mathrsfs}
\usepackage[dvipsnames]{xcolor}
\usepackage{hyperref}
\hypersetup{
    colorlinks=true,
    linkcolor=black, %blue
    citecolor=black} %Maroon
\usepackage{graphicx}
\usepackage{subfig}
\DeclareMathOperator{\ArcCot}{arccot}

\let\originalleft\left
\let\originalright\right
\renewcommand{\left}{\mathopen{}\mathclose\bgroup\originalleft}
\renewcommand{\right}{\aftergroup\egroup\originalright}

\protected\def\verythinspace{%
  \ifmmode
    \mskip0.5\thinmuskip
  \else
    \ifhmode
      \kern0.08334em
    \fi
  \fi
}

\theoremstyle{plain}
\newtheorem{theorem}{Theorem}

\newtheorem{corollary}[theorem]{Corollary}
\theoremstyle{remark}
\newtheorem{remark}[theorem]{Remark}
\newtheorem*{remark*}{Remark}
\newtheorem*{remarks*}{Remarks}
\theoremstyle{definition}

\newtheorem{problem}[theorem]{Problem}
\newtheorem{definition}[theorem]{Definition}
\newtheorem*{notation*}{Notation}
\newtheorem*{example*}{Example}

\begin{document}
\title[Nonrigidity of flat ribbons]{Nonrigidity of flat ribbons}

\author{Matteo Raffaelli}
\address{Institute of Discrete Mathematics and Geometry\\
TU Wien\\
Wiedner Hauptstra{\ss}e 8-10/104\\
1040 Vienna\\
Austria}
\email{matteo.raffaelli@tuwien.ac.at}
\thanks{This work was supported by Austrian Science Fund (FWF) project F 77 (SFB ``Advanced Computational Design'').}
%\date{\today}
\date{August 29, 2022}
\begin{abstract}
	We study ribbons of vanishing Gaussian curvature, i.e., \emph{flat} ribbons, constructed along a curve in $\mathbb{R}^{3}$. In particular, we first investigate to which extent the ruled structure determines a flat ribbon: in other words, we ask whether for a given curve $\gamma$ and ruling angle (angle between the ruling line and the curve's tangent) there exists a well-defined flat ribbon. It turns out that the answer is positive only up to an initial condition, expressed by a choice of normal vector at a point. We then study the set of infinitely narrow flat ribbons along a fixed curve $\gamma$ in terms of energy. By extending a well-known formula for the bending energy of the rectifying developable, introduced in the literature by Sadowsky in 1930, we obtain an upper bound for the difference between the bending energies of two solutions of the initial value problem. We finally draw further conclusions under some additional assumptions on the ruling angle and the curve $\gamma$.%Further conclusions can be drawn under additional assumptions. We first consider two important choices of ruling angle, and then we apply the results to a circular helix.
\end{abstract}
\maketitle
\tableofcontents

\section{Introduction and main results}\label{Intro}

Developable, or \emph{flat}, surfaces in $\mathbb{R}^{3}$ are among the most classical and well-studied objects in differential geometry \cite{lawrence2011, ushakov1999}. They are characterized by having zero Gaussian curvature or, equivalently, by being ruled surfaces with a constant family of tangent planes along each ruling. Our main interest in this article is to study the set of flat surfaces containing a given space curve, or, more precisely, the set of \emph{flat ribbons along $\gamma$}.

Let $I =[0,L]$, let $\gamma \colon I \to \mathbb{R}^{3}$ be a smooth, regular connected curve, and let $S \subset\mathbb{R}^{3}$ be a smooth surface; without loss of generality, we may assume $\gamma$ to be unit-speed. We say that $S$ is \textit{locally nonplanar} if it does not contain any planar open set. Further, if $S$ is ruled and $\gamma(t) \in S$, then we define the \textit{width of $S$ (with respect to $\gamma$) at $t$} to be the length of the projection of the ruling passing from $\gamma(t)$ onto the normal plane $\gamma'(t)^{\perp}$.

\begin{definition}
	A developable surface $D$ that contains $\gamma$ is called a \textit{flat ribbon along $\gamma$} if the following conditions are satisfied:
	\begin{enumerate}
		\item $D$ is locally nonplanar, and it is a compact subset of $\mathbb{R}^{3}$.
		\item $\gamma$ is transversal to every ruling of $D$ and meets each of them at the midpoint.
		\item $D$ has constant width.
	\end{enumerate}
\end{definition}

It is well known that, if the curvature of $\gamma$ is always different from zero, then there exist plenty of flat ribbons along $\gamma$. Indeed, let $N \colon I \to \mathbb{R}^{3}$ be a unit vector field---always normal to $\gamma'$---along $\gamma$. It is not difficult to check that, if $\langle \gamma''(t), N(t) \rangle \neq 0$ for all $t \in I$, then the image of the map $ t \mapsto \gamma(t) + (N(t)^{\perp} \cap N'(t)^{\perp})$ is a well-defined surface in a neighborhood of $\gamma$, both locally nonplanar and flat; see \cite[pp.~195--197]{docarmo1976} and section \ref{FlatRibbon}.% We denote it by $\mathcal{R}(N)$.

On the other hand, to any (singly) ruled surface containing $\gamma$ one can associate a function $\alpha \colon I \to [0,\pi)$, called \textit{ruling angle}, describing the angle between the ruling line and the tangent vector of $\gamma$. Different ruled surfaces along $\gamma$ possessing equal ruling angle could/should be regarded as akin, if not equivalent.

It is therefore natural to consider the following problem.

\begin{problem}%\label{mainPROB}
Given a flat ribbon along $\gamma$, describe the set of all flat ribbons along the same curve $\gamma$ having the given width and ruling angle.
%	Let $\mathcal{R}(N)$ denote any flat ribbon normal to $N$ along $\gamma$. Describe the set of all flat ribbons along $\gamma$ having the same width and ruling angle as $\mathcal{R}(N)$.
\end{problem}

In this paper, we shall see that, under some mild conditions, the set in question is isomorphic to a \emph{full} circle. Indeed, suppose that $N$ is the normal vector of a flat ribbon $\mathcal{R}(N)$ along $\gamma$, and denote the corresponding ruling angle by $\alpha(N)$. Then the following result holds.

\begin{theorem}\label{mainTH}
Suppose that $\gamma$ is locally nonplanar, i.e., its restriction to any open interval is nonplanar, and let $\varphi$ be a smooth function $I \to (0, \pi)$. For any $t_{0} \in I$ and any unit vector $v \in \gamma'(t_{0})^{\perp}$, there exists a flat ribbon $\mathcal{R}(V)$ along $\gamma$ such that $V(t_{0})=v$ and $\alpha(V) = \varphi$.
\end{theorem}

\begin{corollary}\label{mainCOR}
	Suppose that $\gamma$ is locally nonplanar, and let $\mathcal{R}(N)$ be a flat ribbon along $\gamma$. For any $t_{0} \in I$ and any unit vector $v \in \gamma'(t_{0})^{\perp}$, there exists a flat ribbon $\mathcal{R}(V)$ along $\gamma$, having the same width as $\mathcal{R}(N)$, such that $V(t_{0})=v$ and $\alpha(V) = \alpha(N)$.
\end{corollary}

\begin{corollary}\label{secondaryCOR}
	Suppose that $\gamma$ is locally nonplanar. The set of all flat ribbons of any fixed width along $\gamma$ admitting a smooth asymptotic parametrization is isomorphic to $C^{\infty}(I; (0,\pi)) \times \mathbb{S}^{1}$.
\end{corollary}

\begin{remark}
The nonplanarity assumption in Theorem~\ref{mainTH} allows $\gamma$ to have isolated points of vanishing curvature or torsion. It is only needed because we have excluded a planar strip to qualify as a flat ribbon. Indeed, if $\gamma$ is planar, then there exists a vector $v$ for which the corresponding ribbon degenerates into a planar strip.
\end{remark}

\begin{remark}
	The definition of $I$ as a \emph{closed} interval is essential for the validity of the theorem. Suppose for a moment that $I$ is an arbitrary interval. Then Theorem~\ref{mainTH} holds provided the functions $\kappa_{g} \cot(\varphi)$ and $\kappa_{n} \cot(\varphi)$ are bounded; see section \ref{ProofMainTH}. Without this extra hypothesis, we could yet prove the following local statement:
	for any $t_{0} \in I$ and any unit vector $v \in \gamma'(t_{0})^{\perp}$, there exists a neighborhood $I_{0}$ of $t_{0}$ and a flat ribbon $\mathcal{R}_{0}(V)$ along $\gamma \rvert_{I_{0}}$ such that $V(t_{0})=v$ and $\alpha_{0}(V) = \varphi \rvert_{I_{0}}$.
\end{remark}

\begin{remark}
	The extra assumption in Corollary~\ref{secondaryCOR} is needed because the ruling angle of a flat ribbon along $\gamma$, in the presence of planar points, may fail to be differentiable in a nowhere dense set; see \cite{ushakov1996}.
\end{remark}

To the best of the author's knowledge, Theorem~\ref{mainTH} has not appeared in the literature before. This is somewhat surprising, given the classical nature of the subject and the relative simplicity of the proof.% While, strictly speaking, Theorem \ref{mainTH} may not be considered an improvement of the classical existence result \cite[pp.~195--197]{docarmo1976}, it certainly represents a significant addition to it. Not to mention that the statement in \cite{docarmo1976} is based on the restrictive assumption that $\langle \gamma''(t), N(t) \rangle \neq 0$ for all $t \in I$.

The proof of Theorem~\ref{mainTH}, which is based on the standard theory of ordinary differential equations, will be given in section \ref{ProofMainTH}. In particular, the proof offers a means to construct the solution by solving a nonlinear differential equation of first order; see Figure \ref{figureTorus}.% For an alternative perspective on the problem of constructing flat ribbons with a prescribed ruling angle, see \cite{bohr2013}.

\begin{figure}[h]
	\centering
	
	\subfloat[$q=-\pi/2$]{\includegraphics[width=0.4\textwidth]{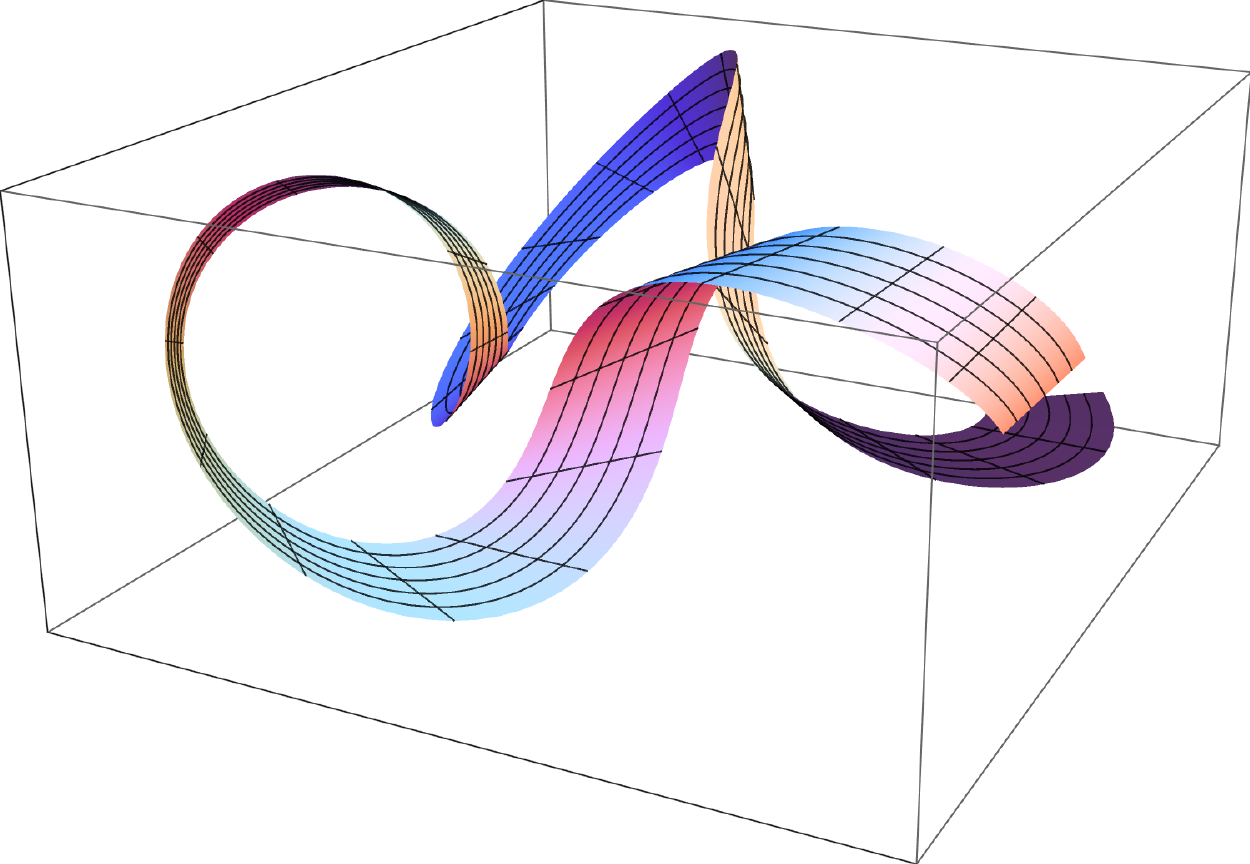}\label{torus1}} \qquad
	\subfloat[$q=-\pi/3$]{\includegraphics[width=0.4\textwidth]{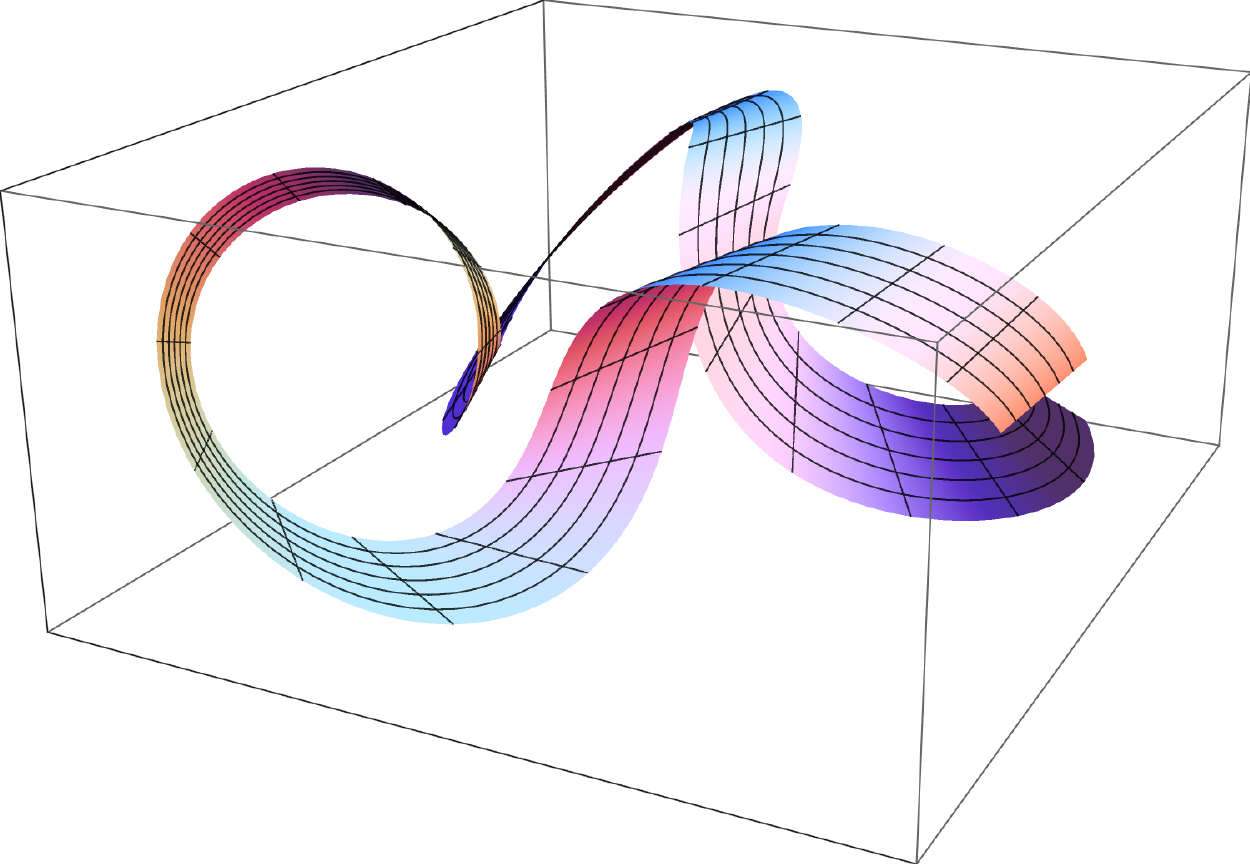}\label{torus2}}
	
	\subfloat[$q=-\pi/6$]{\includegraphics[width=0.4\textwidth]{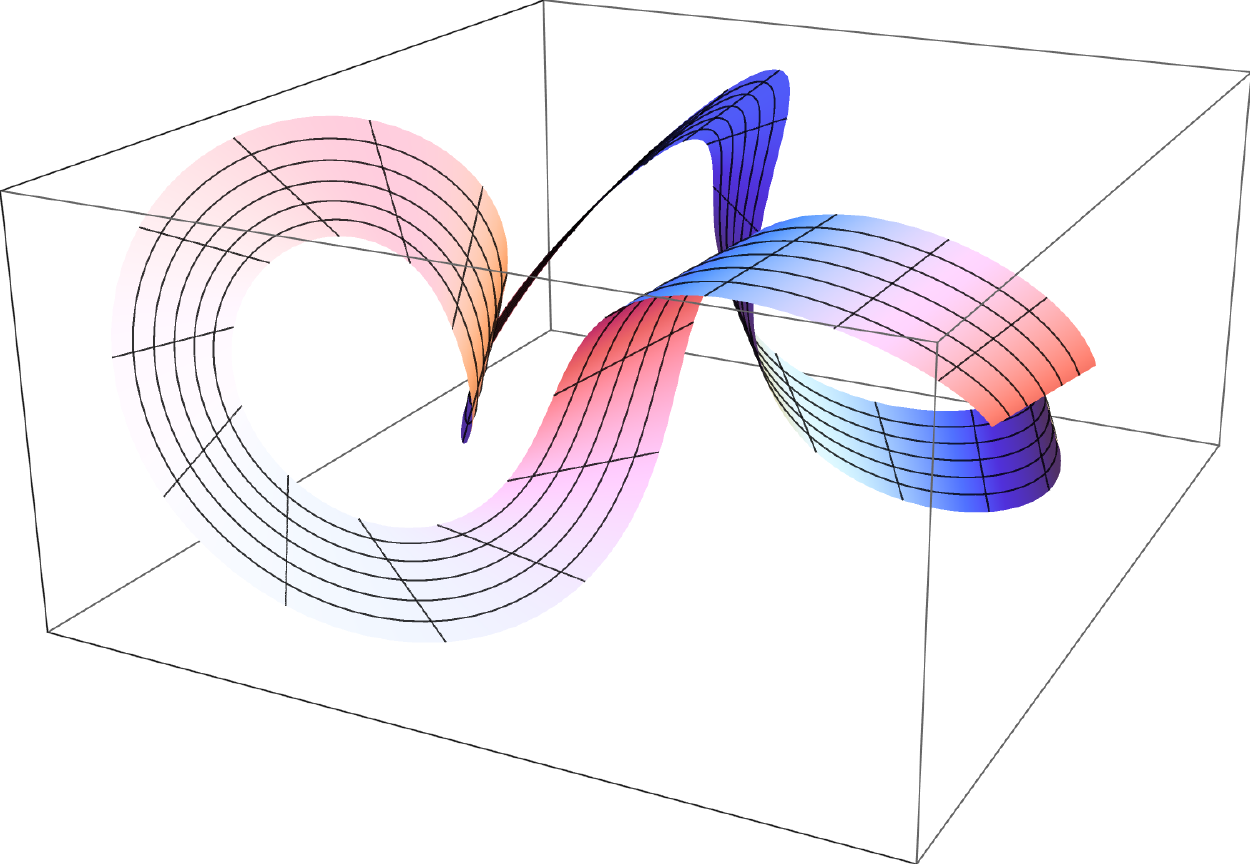}\label{torus3}}\qquad
	\subfloat[$q=0$]{\includegraphics[width=0.4\textwidth]{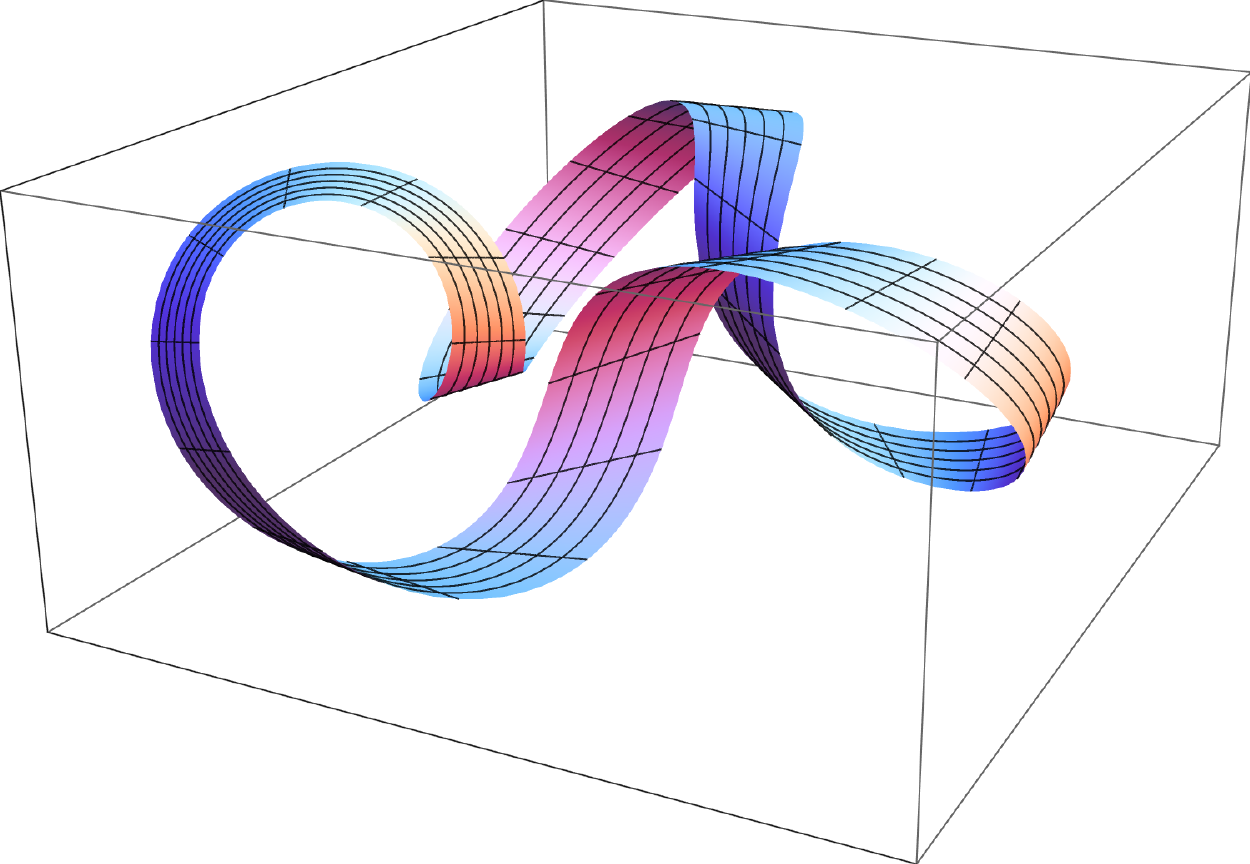}\label{torus4}}
	\caption[A]{Examples of flat ribbons along $\gamma$ having the same width and ruling angle. The curve $\gamma \colon [0,2\pi] \to \mathbb{R}^{3}$ is a trivial torus knot, while the ruling angle is induced by the unit normal vector of the torus; in other words, we are considering the ruling angle of a flat ribbon that is tangent to the torus along $\gamma$ (shown in plot \subref{torus4}). Each plot corresponds to a different initial condition $v \in \gamma'(0)^{\perp}$, obtained by rotating the normal vector of the torus at $\gamma(0)$ by an angle $q$. Plots \subref{torus1}, \subref{torus2}, and \subref{torus3} are generated by solving numerically equation \eqref{newODE2}.}\label{figureTorus}
	%The ribbon in \subref{torus4} is tangent to the torus along $\gamma$.
\end{figure}

It is worth emphasizing that any two flat ribbons $\mathcal{R}(N_{1})$ and $\mathcal{R}(N_{2})$ are \emph{locally} isometric, by Minding's theorem. More precisely, for any $p_{1} \in \mathcal{R}(N_{1})$ and $p_{2} \in \mathcal{R}(N_{2})$, there exist neighborhoods $\mathcal{U}_{1}$ of $p_{1}$, $\mathcal{U}_{2}$ of $p_{2}$ and an isometry $\mathcal{U}_{1} \to \mathcal{U}_{2}$. On the other hand, if $\mathcal{R}(N_{1})$ and $\mathcal{R}(N_{2})$ have the same ruling angle, then in general they are \emph{not} globally isometric. This can be deduced from the fact that the geodesic curvatures of $\gamma$ relative to $\mathcal{R}(N_{1})$ and $\mathcal{R}(N_{2})$ are typically different; see Remark~\ref{kappaGremark}.
%the geodesic curvature of $\gamma$, thought of as a curve in $\mathcal{R}(V_{1})$, is generally different than the one relative to $\mathcal{R}(V_{2})$;

The second objective of the paper is to understand the set of flat ribbons along $\gamma$ in terms of energy. In 1930, Sadowsky \cite{sadowsky1930, hinz2015} argued that the bending energy $\int_{D} H^{2} \,dA$ of the rectifying developable of $\gamma$, in the limit of infinitely small width, is proportional to
\begin{equation*}
\int_{0}^{L} \kappa^{2}\left(1+\mu^{2}\right)^{2} \, dt.
\end{equation*}
Here $\kappa >0$ is the curvature of $\gamma$ and $\mu = -\tau/\kappa$, where $\tau$ is the torsion. Sadowsky's claim was formally justified by Wunderlich \cite{wunderlich1962, todres2015}.%; see \cite{chubelaschwili2010, giomi2010, dias2014, starostin2018} for some recent applications.

In section \ref{ProofMainTH} we will prove that Sadowsky's result extends virtually unchanged to any flat ribbon along $\gamma$; cf.\ \cite{efrati2015}.
\begin{theorem}\label{EnergyTH}
If $\mathcal{R}(N)$ has width $2w$, then its bending energy $\mathcal{E}(\mathcal{R}(N)) =\mathcal{E}(N)$ satisfies
	\begin{equation}\label{EnergyFormula}
		\lim_{w \to 0} \mathcal{E}(N)= \lim_{w \to 0}\frac{w}{2}\int_{0}^{L} \kappa_{n}^{2}\left(1+\cot(\alpha(N))^{2}\right)^{2} \, dt,
	\end{equation} 
	where $\kappa_{n}$ is the normal curvature of $\gamma$ with respect to $N$, as defined in section \ref{DarbouxFrame}.%;see section \ref{DarbouxFrame} for definitions.
\end{theorem}

\begin{remark}
The function $\cot(\alpha(N))$ agrees with $\mu = -\tau_{g}/\kappa_{n}$ on the subset $I \setminus \kappa_{n}^{-1}(\{0\})$, which is dense in $I$; here $\tau_{g}$ is the geodesic torsion of $\gamma$ with respect to $N$. Thus $\cot(\alpha(N))$ is the unique continuous extension of $\mu$ to $I$.
\end{remark}

Theorem~\ref{EnergyTH} tells us that, for any ruling angle, the ribbon in which $\gamma$ has the least energy ($=\int_{0}^{L}\kappa_{g}^{2}\,dt$, where $\kappa_{g}$ is the geodesic curvature of $\gamma$ with respect to $N$) costs the most energy, and vice versa. Hence when $\kappa > 0$ we obtain: \emph{among all infinitely narrow flat ribbons along $\gamma$ having ruling angle $\alpha(T'/\lVert T'\rVert)$, the rectifying developable of $\gamma$ has the maximum bending energy.}

%A more technical observation is contained in the next corollary.
More generally, the following corollary applies.

\begin{corollary}\label{EnergyCOR}
If $\mathcal{R}(N)$ and $\mathcal{R}(V)$ are flat ribbons along $\gamma$ with the same ruling angle, then, in the limit of infinitely small widths, their bending energies $\mathcal{E}(N)$ and $\mathcal{E}(V)$ satisfy
\begin{equation*}
\mathcal{E}(V) \leq \mathcal{E}(N) + \frac{w}{2}\int_{0}^{L} \kappa_{g}^{2}\left(1 + \cot(\alpha(N))^{2}\right)^{2} \, dt,
\end{equation*}
where $\kappa_{g}$ is the geodesic curvature of $\gamma$ with respect to $N$, as defined in section 2.
In particular, if the normal curvature $\kappa_{n}$ of $\gamma$ with respect to $N$ is always nonzero, then 
\begin{equation*}
\frac{\mathcal{E}(V)}{\mathcal{E}(N)} \leq 1 + \max \rho^{2} = 1 + \max \tan (\phi)^{2},
\end{equation*}
where $\rho = \kappa_{g}/\kappa_{n}$ and $\phi$ is the angle between $N$ and $\gamma''$.
\end{corollary}

%In terms of the angle $\phi$ between $N$ and $\gamma''$, Corollary \ref{EnergyCOR} can be equivalently stated as follows. \emph{Under any deformation of $\mathcal{R}(N)$ preserving both $\gamma$ and the ruling angle, the bending energy cannot exceed $\mathcal{E}(N)$ by more than an additive constant equal to $\mathcal{E}(N) \max \tan (\phi)^{2}$}.

The plan of the paper is as follows. The next two sections present the preliminaries needed for the proof of Theorem~\ref{mainTH}, which is carried out in section~\ref{ProofMainTH}. In section~\ref{BendingEnergy} we then proceed with the proofs of Theorem~\ref{EnergyTH} and Corollary~\ref{EnergyCOR}. In the subsequent section we derive further results by considering two natural choices of ruling angle. Finally, in section~\ref{Helix} we specialize the discussion to the case where the curve $\gamma$ is a circular helix.

%There is a vast literature on ribbons from the point of view of physics.
%We mention only \cite{bohr2013, fosdick2016, freddi2018, seguin2020}, whereas interesting applications of Sadowsky's energy formula can be found in \cite{chubelaschwili2010, giomi2010, dias2014, starostin2018}.
This work joins several other recent studies on ribbons; see e.g.\ \cite{bohr2013, fosdick2016, freddi2018, seguin2020}. In particular, the problem of constructing flat surfaces along a given curve has also been considered in \cite{zhao2008, izumiya2015, hananoi2017, raffaelli2018}; interesting applications of Sadowsky's energy formula can be found in \cite{chubelaschwili2010, giomi2010, dias2014, starostin2018}.

In fact, a closely related work \cite{seguin2021} appeared shortly before the first version of this paper was completed. By basing their analysis on the geodesic curvature---rather than the ruling angle---the authors in \cite{seguin2021} offer an alternative description of the surfaces studied here.

\section{The Darboux frame}\label{DarbouxFrame}
We begin by defining the Darboux frame. Classically, that is a natural frame along a surface curve. For our purposes, the surface is not important, only the normal vector is.

Let $N$ be a (smooth) unit normal vector field along $\gamma$, let $T$ be the unit tangent vector $\gamma'$ of $\gamma$, and let $H=N\times T$. We define
\begin{itemize}[label={--}]
\item the \textit{Darboux frame of $\gamma$ with respect to $N$} to be the triple $(T, H, N)$;
\item the \textit{geodesic curvature $\kappa_{g}$ of $\gamma$ with respect to $N$} by $\kappa_{g} = \langle T', H \rangle$;
\item the \textit{normal curvature $\kappa_{n}$ of $\gamma$ with respect to $N$} by $\kappa_{n} = \langle T', N \rangle$;
\item the \textit{geodesic torsion $\tau_{g}$ of $\gamma$ with respect to $N$} by $\tau_{g} = \langle H', N \rangle$.
\end{itemize}

Since $(T,H,N)$ is a frame along $\gamma$, we may express the derivative of any of its elements in terms of the frame itself. In fact, being $(T,H,N)$ orthonormal, it is easy to verify that the following equations hold:
\begin{equation*} %\label{FrameEquation}
\begin{pmatrix}
T' \\
H'\\
N'\\
\end{pmatrix}
= 
\begin{pmatrix}
0 & \kappa_{g} & \kappa_{n}\\
-\kappa_{g} & 0 & \tau_{g}\\
-\kappa_{n} & -\tau_{g} & 0\\
\end{pmatrix}
\begin{pmatrix}
T \\
H\\
N\\
\end{pmatrix} .
\end{equation*}

\section{Constructing a flat ribbon}\label{FlatRibbon}

The Darboux frame is a useful tool for constructing a flat ribbon normal to $N$ along $\gamma$, in that it permits to prescribe its width, which by definition is measured along the vector field $H$.%The main existence result is contained in the following theorem.

\begin{theorem}[{\cite[pp.~195--197]{docarmo1976}}, \cite{izumiya2015, raffaelli2018}]\label{existenceTH}
Suppose that $\kappa_{n}(t) \neq 0$ for all $t\in I$. Then there exists $w>0$ and a unique flat ribbon of width $2w$ normal to $N$ along $\gamma$. Such ribbon is parametrized by $\sigma \colon I \times [-w,w] \to \mathbb{R}^{3}$,
\begin{equation*}
\sigma(t,u) = \gamma(t) +u \left( H(t) +\mu(t)T(t)\right),
\end{equation*}
where $\mu=-\tau_{g}/\kappa_{n}$. Conversely, if $\mathcal{R}(N)$ is a flat ribbon normal to $N$ along $\gamma$, then $\tau_{g}(t) = 0$ for all $t \in I$ such that $\kappa_{n}(t)=0$.
\end{theorem}

For the reader's convenience, we give a short proof of the theorem.

\begin{proof}[Proof of Theorem~\ref{existenceTH}]
Given a vector field $X$ along $\gamma$, let $\sigma$ be defined by
\begin{align*}
	\sigma \colon I \times \mathbb{R} &\to \mathbb{R}^{3}\\
	(t, u) &\mapsto \gamma(t) + u X(t).
\end{align*}
Recall that $\sigma$ is flat exactly when $T$, $X$, and $X'$ are everywhere linearly dependent; accordingly, we need to find $X$ such that
\begin{equation}\label{SystemX}
	\begin{cases}
		\langle X(t), N(t) \rangle = 0,\\
		\langle X(t) \times T(t), X'(t)\rangle =0,\\
		X(t) \times T(t) \neq 0
	\end{cases}
\end{equation}
for all $t \in I$. Note that the first two equations are equivalent to $X(t) \in N(t)^{\perp} \cap N'(t)^{\perp}$. 

Suppose that $\kappa_{n}(t) \neq 0$. Then $N'(t) = -\kappa_{n}(t)T(t) -\tau_{g}(t)H(t) \neq 0$; the intersection $N(t)^{\perp} \cap N'(t)^{\perp}$ has dimension one and is spanned by 
		\begin{equation*}
N'(t) \times N(t) = \kappa_{n}(t)H(t) - \tau_{g}(t)T(t),
	\end{equation*}
as desired. 

Conversely, suppose that $\mathcal{R}(N)$ is a flat ribbon normal to $N$ along $\gamma$. Then $\mathcal{R}(N)$ lies in the image of $\sigma$ for some $X$ satisfying \eqref{SystemX}. Hence $\tau_{g}(t) = 0$ whenever $\kappa_{n}(t) =0$, because otherwise $\mathcal{R}(N)$ would be singular at $\gamma(t)$.
\end{proof}

\begin{remark}
If $\kappa_{n}(t) \neq 0$, then the ruling angle $\alpha(N)$ of $\mathcal{R}(N)$ satisfies
\begin{equation}\label{RulingAngleFormula}
\alpha(N(t)) = \ArcCot \left(-\tau_{g}(t)/\kappa_{n}(t)\right).
\end{equation}
\end{remark}

\begin{remark}
	In the spirit of \cite{nomizu1959, randrup1996}, the existence condition in Theorem~\ref{existenceTH} can be weakened as follows. For all $t \in I$, we require that
	\begin{enumerate}[label=(\roman*)]
		\item there exists $l \in \mathbb{N}_{0}$ such that the $l$th derivative $\kappa_{n}^{(l)}$ is nonzero at $t$. This implies, in particular, that every zero of $\kappa_{n}$ is isolated;
		\item $\tau_{g}^{(0)}(t) = \dotsb = \tau_{g}^{(l-1)}(t) = 0$.
	\end{enumerate}

These two conditions guarantee that, if $\kappa_{n}(t)=0$, then $\lim_{z\to t} \tau_{g}(z)/\kappa_{n}(z)$ is well-defined. In fact, it is not difficult to verify that the continuous extension of $\tau_{g}/\kappa_{n}$ to $I$---obtained by setting $\tau_{g}(t)/\kappa_{n}(t) =\lim_{z\to t} \tau_{g}(z)/\kappa_{n}(z)$ whenever $\kappa_{n}(t)=0$---is smooth.
	\end{remark}

\section{Proof of Theorem~\ref{mainTH}}\label{ProofMainTH}

We are now ready to prove Theorem~\ref{mainTH}.

Given any unit normal vector field $N$ along $\gamma$, let the Darboux frame of $\gamma$ with respect to $N$ rotate around the tangent $T$ by a smooth function $\theta \colon I \to \mathbb{R}$:
\begin{align*}
H(\theta) &= \cos (\theta) H + \sin(\theta) N,\\
N(\theta) &= -\sin (\theta) H + \cos (\theta) N.
\end{align*}

The normal curvature of $\gamma$ with respect to $N(\theta)$ is given by
\begin{align}
	\kappa_{n}(\theta)&= \left< T', N(\theta) \right>\nonumber \\
	&= \left< T', -\sin (\theta) H + \cos (\theta) N \right>\nonumber \\
	&= -\kappa_{g}\sin (\theta)  + \kappa_{n}\cos(\theta).\label{kappaNtheta}
\end{align}

Similarly, the geodesic torsion of $\gamma$ with respect to $N(\theta)$ is given by
\begin{align*}
	\tau_{g}(\theta)&= \left< H'(\theta), N(\theta) \right> \\
	&=  \left< H'(\theta), -\sin (\theta) H + \cos (\theta) N  \right>.
\end{align*}
%Writing $\tilde{\kappa}_{n}$, $\tilde{\tau}_{g}$, and $\tilde{\kappa}_{g}$ as shorthands for $\lVert \gamma' \rVert \kappa_{n}$, $\lVert \gamma' \rVert \tau_{g}$, and $\lVert \gamma' \rVert \kappa_{g}$, respectively,
We first compute
\begin{align*}
	H'(\theta) &= \theta'\cos (\theta) N + \sin(\theta) N' -\theta' \sin(\theta) H + \cos(\theta)H'\\
	&= \theta'\cos (\theta) N - \sin(\theta) \left(\kappa_{n}T +\tau_{g}H\right)-\theta' \sin(\theta) H + \cos(\theta) \left(- \kappa_{g} T + \tau_{g}N \right)\\
	&=  -\left(\cos(\theta)\kappa_{g} +\sin(\theta) \kappa_{n}\right) T -\sin(\theta) \left( \theta' +\tau_{g}\right) H +\cos(\theta) \left( \theta' + \tau_{g}\right) N.
\end{align*}
Let us also compute
\begin{equation*}
 -\sin (\theta) \left< H'(\theta), H\right> =\sin (\theta)^{2} \left(\theta' + \tau_{g}\right),
\end{equation*}
and
\begin{equation*}
\cos (\theta) \left< H'(\theta), N\right> =\cos (\theta)^{2} \left(\theta' + \tau_{g}\right).
\end{equation*}
It follows that
\begin{equation}\label{tauGtheta}
\tau_{g}(\theta)=\theta' + \tau_{g}.
\end{equation}

Next, let $\varphi$ be a smooth function $I \to (0,\pi)$, and assume that $\gamma$ is locally nonplanar. %and that its normal curvature with respect to $N$ never vanishes, so that the ribbon $\mathcal{R}(N)$ is well-defined.
We claim that, if the condition
\begin{equation}\label{ODE}
\kappa_{n}(\theta)\cot(\varphi)  +\tau_{g}(\theta)=0
\end{equation}	
holds, then the ribbon $\mathcal{R}(N(\theta))$ is well-defined, and its ruling angle is exactly $\varphi$. To verify the claim, note that the zero set of $\kappa_{n}(\theta)$ is nowhere dense in $I$, because otherwise there would be an interval where both $\kappa_{n}(\theta)$ and $\tau_{g}(\theta)$ vanish, contradicting the assumption of local nonplanarity; thus the ruling angle $\alpha(N(\theta))$ is well-defined on a dense subset, where it agrees with $\varphi$, and so it admits a unique continuous extension to the entire interval.
%Indeed, the set of points where $\kappa_{n}(\theta)=0$ is discrete in $I$, because otherwise there would be an interval where both $\kappa_{n}(\theta)$ and $\tau_{g}(\theta)$ vanish. In turn, this would imply that the restriction of $\gamma$ to such interval is planar.

Substituting the expressions of $\kappa_{n}(\theta)$ and $\tau_{g}(\theta)$ obtained earlier, condition \eqref{ODE} becomes
\begin{equation}\label{newODE}
\theta' + \cot(\varphi) \left( \kappa_{n} \cos(\theta) -  \kappa_{g}\sin(\theta) \right) +\tau_{g} = 0.
\end{equation}
This is a first-order, nonlinear ordinary differential equation in $\theta \colon I \to \mathbb{R}$, which admits a unique local solution for any initial condition $\theta(t_{0}) = q \in [0, 2\pi)$. %($t_{0} \in I$).

It remains to check that the initial value problem is globally solvable, that is, its solution can be extended to the entire interval $I$.

Define $F \colon I \times \mathbb{R}$ by
\begin{equation*}
	F(t,x) = \cot(\varphi) \left(\kappa_{g} \sin(x) - \kappa_{n}\cos(x) \right) -\tau_{g}.
\end{equation*}
We are going to show that $F$ satisfies the following Lipschitz condition: there exists a constant $c >0$ such that, for every $t \in I$ and every $x,y \in \mathbb{R}$,
\begin{equation*}
	\left\lvert F(t,x)-F(t,y) \right\rvert \leq c \left\lvert x-y \right\rvert.
\end{equation*}
This way the statement will follow from the classical Picard--Lindel\"{o}f theorem; see e.g.\ \cite[Theorem~3.1]{o'regan1997}.

First of all, note that both $\kappa_{g}\cot(\varphi)$ and $\kappa_{n}\cot(\varphi)$ are bounded, because they are continuous on the closed interval $I$. Let $l$ and $m$ be upper bounds for $\lvert\kappa_{g}\cot(\varphi)\rvert$ and $\lvert\kappa_{n}\cot(\varphi)\rvert$, respectively. Computing
\begin{align*}
&\left\lvert F(t,x)-F(t,y) \right\rvert\\
&\qquad = \left\lvert \kappa_{g}(t) \cot(\varphi(t))  \left( \sin (x) - \sin (y)\right) + \kappa_{n}(t)\cot(\varphi(t))  \left( \cos(y) - \cos(x) \right)\right\rvert\\
&\qquad \leq \left\lvert \kappa_{g}(t) \cot(\varphi(t))\left( \sin (x) - \sin (y)\right) \right\rvert + \left\lvert \kappa_{n}(t)\cot(\varphi(t)) \left( \cos(y) - \cos(x) \right) \right\rvert\\
&\qquad = \left\lvert \kappa_{g}(t)\cot(\varphi(t)) \right\rvert \left\lvert  \sin (x) - \sin (y)\right\rvert + \left\lvert \kappa_{n}(t) \cot(\varphi(t)) \right\rvert \left\lvert  \cos(x) - \cos(y) \right\rvert,
\end{align*}
we observe that
\begin{equation*}
	\left\lvert F(t,x)-F(t,y) \right\rvert \leq  (l + m) \left\lvert x - y \right\rvert,
\end{equation*}
and so $F$ satisfies the Lipschitz condition with $c = l+m$, as desired.

\begin{remark}\label{ODEremark}
	It follows from section \ref{FlatRibbon} that $\mathcal{R}(N(\theta))$ has the same ruling angle as $\mathcal{R}(N)$ if and only if
	\begin{equation}\label{ODE2}
		\tau_{g}\kappa_{n}(\theta)=\tau_{g}(\theta)\kappa_{n}.
	\end{equation}
	By substituting \eqref{kappaNtheta} and \eqref{tauGtheta}, we observe that \eqref{ODE2} is equivalent to
	\begin{equation}\label{newODE2}
	\kappa_{n}\theta'+ \kappa_{g}\tau_{g} \sin (\theta) - \kappa_{n}\tau_{g}\cos(\theta)+\kappa_{n}\tau_{g} = 0.
	\end{equation}
Compared with \eqref{newODE}, equation \eqref{newODE2} offers a shortcut to the construction of the ribbon $\mathcal{R}(V)$ defined in Corollary~\ref{mainCOR}.% We will solve it, under some additional assumption on the curve $\gamma$, in section \ref{SpecialCases}.
\end{remark}

\begin{remark}\label{kappaGremark}
	The geodesic curvature of $\gamma$ with respect to $N(\theta)$ is given by
	\begin{align*}
		\kappa_{g}(\theta)&= \left< T', H(\theta) \right> \\
		&= \left< T', \cos (\theta) H + \sin(\theta) N \right> \\
		&= \kappa_{g}\cos (\theta)  + \kappa_{n}\sin(\theta).
	\end{align*}

It is easy to see that $\kappa_{g}(\theta)=\kappa_{g}$ if and only if $N(\theta) = N(\pi -2 \bar{\theta})$, where $\bar{\theta}$ satisfies $\sin(\bar{\theta}) = \kappa_{g}/\kappa$ and $\cos(\bar{\theta}) = \kappa_{n}/\kappa$. It follows that, if $\kappa>0$, then there exist pairs of \emph{globally isometric} flat ribbons along $\gamma$; cf.\ \cite{fuchs1999}. This is in striking contrast to the case of positive Gaussian curvature, where a surface is globally rigid relative to any of its curves \cite{ghomi2020}.
\end{remark}

\section{Bending energy}\label{BendingEnergy}
Let $D$ be a flat surface in $\mathbb{R}^{3}$. The bending energy $\mathcal{E}(D)$ of $D$ is defined by
\begin{equation*}
	\mathcal{E}(D)= \int_{D} H^{2} \, dA,
\end{equation*}
where $H$ is the mean curvature and $dA$ the area element of $D$.

The purpose of this section is to prove Theorem~\ref{EnergyTH} and Corollary~\ref{EnergyCOR} in the introduction.

\begingroup
\def\thetheorem{\ref*{EnergyTH}}
\begin{theorem}
If $\mathcal{R}(N)$ has width $2w$, then its bending energy $\mathcal{E}(\mathcal{R}(N)) =\mathcal{E}(N)$ satisfies
\begin{equation*}
	\lim_{w \to 0} \mathcal{E}(N)= \lim_{w \to 0}\frac{w}{2}\int_{0}^{L} \kappa_{n}^{2}\left(1+\cot(\alpha(N))^{2}\right)^{2} \, dt.
\end{equation*} 
%where $\mu = -\tau_{g}/\kappa_{n}$.
\end{theorem}
\addtocounter{theorem}{-1}
\endgroup

\begin{proof}
Since the integrand is zero whenever $\kappa_{n}(t)=0$, we may assume that $\kappa_{n}$ is never zero. We need to prove that
\begin{equation*}
	\lim_{w \to 0} \mathcal{E}(N)= \lim_{w \to 0}\frac{w}{2}\int_{0}^{L} \kappa_{n}^{2}\left(1+\mu^{2}\right)^{2} \, dt,
\end{equation*} 
where $\mu = -\tau_{g}/\kappa_{n}$.

Our first goal is to compute the expressions of the mean curvature and the area element of $\mathcal{R}(N)$ in the standard parametrization $\sigma \colon [0,L] \times [-w,w] \to \mathbb{R}^{3}$,
\begin{equation*}
\sigma(t,u)=\gamma(t) + u X(t), \quad X(t) = \mu(t)T(t)+H(t).
\end{equation*}
This way, we will obtain a formula for the bending energy of a finite-width ribbon $\mathcal{R}(N)$ along $\gamma$.

%As in the proof of Theorem \ref{mainTH}, let $\tilde{\kappa}_{n}$, $\tilde{\tau}_{g}$, and $\tilde{\kappa}_{g}$ be shorthands for $\lVert \gamma' \rVert \kappa_{n}$, $\lVert \gamma' \rVert \tau_{g}$, and $\lVert \gamma' \rVert \kappa_{g}$, respectively. 
As the reader may verify, the components of the first and second fundamental forms are
\begin{align*}
E &=  \left< \gamma' + u X', \gamma' + u X'\right> =\left(1 + u\left(\mu' - \kappa_{g}\right)\right)^{2} + \left(u\mu\kappa_{g}\right)^{2},\\
F &= \left< \gamma' + u X', X\right>= \mu\left( 1 + u \mu'\right),\\
G &= \left< X, X\right> =1+\mu^{2},
\end{align*}
and
\begin{align*}
e &= \left< \gamma'' +u X'', N \right> = \kappa_{n}\left( 1 + u \left(\mu'- \kappa_{g}-\kappa_{g}\mu^{2}\right)\right),\\
f &=  \left< X', N \right> =0,\\
g &= \left< 0, N \right> =0,
\end{align*}
respectively. A computation reveals that the area element is given by
\begin{equation*}
\sqrt{EG-F^{2}}= 1 +u\mu'- u\left(1+\mu^{2}\right)\kappa_{g},
\end{equation*}
whereas for the mean curvature one obtains
\begin{equation*}
H= \frac{G e}{2\left(EG-F^{2}\right)} =- \frac{\left(1+\mu^{2}\right)\kappa_{n}}{2\left(1 +u\mu'-u \left(1+\mu^{2}\right)\kappa_{g}\right)}.
\end{equation*}

The bending energy may therefore be computed by
\begin{align*}
\mathcal{E}(N) &= \int_{0}^{L} \int_{-w}^{w} H^{2} \sqrt{EG-F^{2}} \, du\, dt\\
&=\frac{1}{4} \int_{0}^{L} \int_{-w}^{w}  \frac{\left(1+\mu^{2}\right)^{2}\kappa_{n}^{2}}{1 +u\left(\mu'- \left(1+\mu^{2}\right)\kappa_{g}\right)}\, du \, dt.
\end{align*}
In particular, in the closed subset where $\mu'= (1+\mu^{2})\kappa_{g}$, the integrand does not depend on $u$, and so the inner integral reduces to
\begin{equation}\label{IntegralSpecialCase}
2w\kappa_{n}^{2}\left(1+\mu^{2}\right)^{2} .
\end{equation}
Hence, we may assume that $\mu'(t)\neq (1+\mu(t)^{2})\kappa_{g}(t)$ for every $t \in I$. In that case, integration with respect to $u$ gives
\begin{equation}\label{EnergyFiniteWidth}
\mathcal{E}(N) = \frac{1}{4} \int_{0}^{L} \frac{\left(1+\mu^{2}\right)^{2}\kappa_{n}^{2}}{\mu'-\left(1+\mu^{2}\right) \kappa_{g}} \log \left(\frac{1 +w\left(\mu'- \left(1+\mu^{2}\right)\kappa_{g}\right)}{1 -w\left(\mu'- \left(1+\mu^{2}\right)\kappa_{g}\right)} \right)\, dt.
\end{equation}

Our task is now to evaluate the limit of $\mathcal{E}(N)$ as $w$ approaches zero. We first rewrite \eqref{EnergyFiniteWidth} by means of the following notations:
\begin{align}
	\lambda &= \mu'- (1+\mu^{2})\kappa_{g}, \nonumber \\
	\eta_{1} &= \frac{1}{w \lambda} \log \left(\frac{1 +w\lambda}{1 -w\lambda} \right),\nonumber \\
	\eta_{2}&= w \left(1+\mu^{2}\right)^{2}\kappa_{n}^{2},\nonumber \\
	\mathcal{E}(N) &=  \frac{1}{4} \int_{0}^{L} \eta_{1} \eta_{2}\, dt. \label{EnergyFiniteWidthFG}
	\end{align}
Since $\eta_{1}$ converges pointwise to $2$ as $w \to 0$, it is clear that the integrand $\eta_{1}\eta_{2}$ converges pointwise to $0$ as $w \to 0$. In fact, we will show that the convergence is uniform. Therefore, it will follow from standard analysis \cite[p.~251]{bartle2011} that
\begin{align*}
\lim_{w\to 0} \mathcal{E}(N)&=\frac{1}{4} \int_{0}^{L} \lim_{w\to 0}  \eta_{1}\eta_{2} \, dt\\
&= \frac{1}{4} \int_{0}^{L} \lim_{w\to 0}  2 w\left(1+\mu^{2}\right)^{2}\kappa_{n}^{2}\, dt\\
&=\lim_{w \to 0}\frac{w}{2}\int_{0}^{L} \left(1+\mu^{2}\right)^{2} \kappa_{n}^{2} \, dt.
\end{align*}

It is evident that $\eta_{2}$ is uniformly convergent to $0$ as $w \to 0$. Thus, being $\eta_{1}$ and $\eta_{2}$ bounded, proving that $\eta_{1}$ converges uniformly to $2$ is sufficient to establish uniform convergence of $\eta_{1}\eta_{2}$; see \cite[p.~247]{bartle2011}.

Let 
\begin{equation}\label{hDef}
\xi = \frac{1}{w \lambda} \log \left(\frac{1 +w\lambda}{1 -w\lambda} \right) - 2.
\end{equation}
We need to check that $\max_{t\in I}\lvert \xi(t) \rvert \to 0$ as $w \to 0$. To this end, we calculate $\xi'$ and set it equal to $0$. This leads to
\begin{equation*}
-2 w \lambda \lambda'+\log \left(\frac{1 +w\lambda}{1 -w\lambda} \right) \left( \lambda'-w^{2}\lambda^{2}\lambda' \right)=0.
\end{equation*}

Note that the term multiplying the logarithm only vanishes when $\lambda'$ does. Let $J$ be the zero set of $\lambda'$. Then, since $J$ is independent of $w$, it follows that
\begin{equation*}
\lim_{w\to 0} \max_{ J} \left\lvert \xi\right\rvert =0.
\end{equation*}

On the other hand, in the subinterval $(0,L) \setminus J$, we have $\xi'=0$ if and only if
\begin{equation*}
\log \left(\frac{1 +w\lambda}{1 -w\lambda} \right)= \frac{2 w  \lambda \lambda'}{ \lambda' - w^{2}\lambda^{2}\lambda'}.
\end{equation*}
Together with \eqref{hDef}, this implies
\begin{equation*}
\max_{ (0,L)\setminus J}	\left\lvert \xi\right\rvert =\max_{ (0,L)\setminus J} \left\lvert \frac{2 \lambda'}{\lambda' -w^{2}\lambda^{2}\lambda'}- 2\right\rvert,
\end{equation*}
from which we observe that
\begin{equation*}
\lim_{w\to 0} \max_{ (0,L)\setminus J}	\left\lvert \xi\right\rvert =0.
\end{equation*}

Hence,
\begin{equation*}
\lim_{w\to 0} \eta_{1} = 2 \text{ uniformly},
\end{equation*}
which is the desired conclusion.
\end{proof}

\begin{remark}
If $\gamma$ is locally nonplanar, then, as a function on the set of infinitely narrow flat ribbons along $\gamma$, the bending energy is unbounded. This is because
\begin{equation*}
\lim_{\kappa_{n}(t) \to 0} \kappa_{n}^{2}(t)\left(1+\mu(t)^{2} \right)^{2} = \infty
\end{equation*}
when $\tau_{g}(t) \neq 0$, and one can construct flat ribbons along $\gamma$ of arbitrarily small normal curvature. Indeed, on a subinterval where $\kappa>0$, for $x \in \mathbb{R}$, let $N(x)$ be defined by
\begin{equation*}
\left\lVert T' \right\rVert	N(x)=\cos(x)T' -\sin(x) T'\times T.
	\end{equation*}
%It follows that the normal curvature of $\mathcal{R}(N(x))$ tends to $0$ as $x \to \pi/2$.
It follows that $\tau_{g}(x) = \tau$ and $\kappa_{n}(x) \to 0$ as $x \to \pi/2$.

On the other hand, under the same assumption of local nonplanarity, the bending energy has a \emph{positive} lower bound. Thus, one may search for the ruling angle and initial condition in $\gamma'(0)^{\perp}$ that give the least bending energy. This is a very interesting problem, which the author hopes will be the subject of future study.
\end{remark}

\begingroup
\def\thetheorem{\ref*{EnergyCOR}}
\begin{corollary}
If $\mathcal{R}(N)$ and $\mathcal{R}(V)$ are flat ribbons along $\gamma$ with the same ruling angle, then, in the limit of infinitely small widths, their bending energies $\mathcal{E}(N)$ and $\mathcal{E}(V)$ satisfy
\begin{equation*}
	\mathcal{E}(V) \leq \mathcal{E}(N) + \frac{w}{2}\int_{0}^{L} \kappa_{g}^{2}\left(1 + \cot(\alpha(N))^{2}\right)^{2} \, dt,
\end{equation*}
where $\kappa_{g}$ is the geodesic curvature of $\gamma$ with respect to $N$.
In particular, if the normal curvature $\kappa_{n}$ of $\gamma$ with respect to $N$ is always nonzero, then 
\begin{equation*}
	\frac{\mathcal{E}(V)}{\mathcal{E}(N)} \leq 1 + \max \rho^{2} = 1 + \max \tan (\phi)^{2},
\end{equation*}
where $\rho = \kappa_{g}/\kappa_{n}$ and  $\phi$ is the angle between $N$ and $\gamma''$.
\end{corollary}
\addtocounter{theorem}{-1}
\endgroup

\begin{proof}
Let $\mathcal{R}(N)$ and $\mathcal{R}(V=N(\theta))$ be flat ribbons along $\gamma$ with the same ruling angle. By Theorem~\ref{EnergyTH}, the bending energy of $\mathcal{R}(V)$ satisfies
\begin{align}
	\lim_{w \to 0} \mathcal{E}(V)&= \lim_{w \to 0}\frac{w}{2}\int_{0}^{L} \kappa_{n}(\theta)^{2}\left(1+\cot(\alpha(N(\theta)))^{2}\right)^{2} \, dt \nonumber\\
	&= \lim_{w \to 0}\frac{w}{2}\int_{0}^{L} \kappa_{n}(\theta)^{2}\left(1+\cot(\alpha(N))^{2}\right)^{2} \, dt. \label{EnergySmallWidth}
\end{align} 
%the last equality following directly from \eqref{ODE2}. Moreover, by 
Since $\kappa_{n}(\theta)^{2} \leq \kappa^{2} = \kappa_{n}^{2} + \kappa_{g}^{2}$, assuming that $w$ is infinitely small, we obtain
\begin{align*}
\mathcal{E}(V) &\leq \frac{w}{2}\int_{0}^{L}\kappa_{n}^{2}\left(1+\cot(\alpha(N))^{2}\right)^{2} \, dt + \frac{w}{2}\int_{0}^{L}\kappa_{g}^{2}\left(1+\cot(\alpha(N))^{2}\right)^{2} \, dt\\
&= \mathcal{E}(N) + \frac{w}{2}\int_{0}^{L}\kappa_{g}^{2}\left(1+\cot(\alpha(N))^{2}\right)^{2} \, dt. 
\end{align*}

%\begin{equation}\label{EnergySmallWidth}
%\mathcal{E}(V) =\frac{w}{2}\int_{0}^{L} \left(\kappa_{n}\cos(\theta) - \kappa_{g}\sin(\theta) \right)^{2}\left(1+\cot(\alpha(N))^{2}\right)^{2} \, dt.
%\end{equation} 

Now, suppose that $\kappa_{n}(t) \neq 0$ for all $t \in I$, and let $\rho = \kappa_{g}/\kappa_{n}$. Then
\begin{equation*}
	\mathcal{E}(V) \leq \mathcal{E}(N) + \frac{w}{2}\int_{0}^{L} \rho^{2}\kappa_{n}^{2} \left(1+\cot(\alpha(N))^{2}\right)^{2} \, dt.
\end{equation*}
Noting that the integrand in the equation above is a product of nonnegative functions, by invoking the first mean value theorem for integrals \cite[p.~301]{bartle2011}, we deduce that there exists $s \in I$ such that
\begin{equation*}
\int_{0}^{L} \rho^{2}\kappa_{n}^{2} \left(1+\cot(\alpha(N))^{2}\right)^{2} \, dt = \rho(s)^{2} \int_{0}^{L}\kappa_{n}^{2} \left(1+\cot(\alpha(N))^{2}\right)^{2} \, dt.
\end{equation*}
Hence
\begin{equation*}
\mathcal{E}(V) \leq \mathcal{E}(N) + \rho(s)^{2} \mathcal{E}(N),
\end{equation*}
and the assertion of the corollary follows.
\end{proof}

\section{Special cases}\label{SpecialCases}

In this section we study the set of flat ribbons along a locally nonplanar curve $\gamma$ under two natural choices of ruling angle $\alpha$:
\begin{enumerate}[label=(\Alph*)]
	\item \label{item1} $\alpha$ is constant and equal to $\pi/2$.
	\item \label{item2} Assuming $\kappa>0$, $\alpha$ coincides with the ruling angle of the rectifying developable of $\gamma$.
\end{enumerate}

%To begin with, let $\mathcal{R}(N)$ be a flat ribbon along $\gamma$.

\subsection*{Case \ref{item1}}

Note from \eqref{RulingAngleFormula} that $\alpha(N)=\pi/2$ if and only if $\tau_{g}=0$. Hence, between any two consecutive zeros of $\kappa_{n}$, equation \eqref{newODE2} reduces to $\theta'=0$. Assuming $\kappa_{n}(t) \neq 0$ everywhere, it follows that the initial value problem (defined by $\theta(0)=q$) has constant solution $\theta=q$.

Applying \eqref{EnergySmallWidth} and \eqref{kappaNtheta}, we then observe that the bending energy of $\mathcal{R}(N(q))$, under the hypothesis of infinitely small width, is given by
\begin{equation}\label{EnergySmallWidthCase1}
	\mathcal{E}(N(q)) =\frac{w}{2}\int_{0}^{L} \left(\kappa_{n}\cos(q) - \kappa_{g}\sin(q) \right)^{2} \, dt.
\end{equation} 

In order to analyze the dependence of the bending energy on the initial condition, we calculate $d\mathcal{E}(N(q))/d q$ and set it equal to zero. This leads to
\begin{equation}%\label{qEquation}
A \sin(q)\cos(q) +B\left(\sin(q)^{2}-\cos(q)^{2} \right)=0,
\end{equation} 
where
\begin{align*}
A &= \int_{0}^{L}\left( \kappa_{g}^{2}- \kappa_{n}^{2}\right) \, dt,\\
B &= \int_{0}^{L} \kappa_{g} \kappa_{n} \, dt.
\end{align*}

First, suppose that $A=B=0$. Then the energy is independent of $q$, and
\begin{equation*}
	\mathcal{E}(N(q)) = \frac{w}{4} \int_{0}^{L} \kappa^{2} \, dt.
\end{equation*}
Else, if $B=0$ and $A \neq 0$, then $d\mathcal{E}(N(q))/d q=0$ if and only if $\sin(q)=0$ or $\cos(q)=0$, implying that the extreme values of $\mathcal{E}(N(q))$ are
\begin{equation*}
\frac{w}{2}\int_{0}^{L} \kappa_{n}^{2} \, dt = \mathcal{E}(N)\quad \text{and} \quad \frac{w}{2}\int_{0}^{L} \kappa_{g}^{2} \, dt.
\end{equation*} 

Next, suppose that $B \neq 0$. Noting that $d\mathcal{E}(N(q))/dq\neq0$ if $B\neq0$ and $\sin(q)=0$,  we may assume that $B\neq0$ and $\sin(q)\neq0$. Consequently, $d\mathcal{E}(N(q))/d q$ vanishes if and only if 
\begin{equation*}
B\cot(q)^{2}-A\cot(q) -B=0,
\end{equation*} 
which have solutions
\begin{equation}\label{qSolutions}
q = \ArcCot \frac{A \pm \sqrt{A^{2}+4B^{2}}}{2B}.
\end{equation}
Substitution of \eqref{qSolutions} into \eqref{EnergySmallWidthCase1}, alongside an easy, if tedious, calculation, demonstrates that
\begin{align*}
\max _{q \in [0,2\pi)} \mathcal{E}(N(q)) &=\frac{w}{4} \int_{0}^{L}\kappa^{2} \, dt + \frac{w}{4}\sqrt{A^{2}+4B^{2}},\\
\min _{q \in [0,2\pi)} \mathcal{E}(N(q)) &= \frac{w}{4} \int_{0}^{L} \kappa^{2} \, dt - \frac{w}{4}\sqrt{A^{2}+4B^{2}}.
\end{align*}

\subsection*{Case \ref{item2}}
It is clear that, if $N= T'/\lVert T'\rVert$, then $\kappa_{g}=0$, $\kappa_{n}=\kappa$, and $\tau_{g}=\tau$. Hence in this case equation \eqref{newODE2} simplifies to the separable equation
\begin{equation*}%\label{geodesicODE}
\theta'   =\tau \left(\cos(\theta)-1 \right).
\end{equation*}
Letting $\psi(t)= \int_{0}^{t} \tau(z) \, dz$, the solution is
\begin{align}\label{CaseBSolution}
		\theta_{q}=
	\begin{cases}
0 \quad &\text{if } q=0,\\
2 \ArcCot \left(\cot(q/2) + \psi\right) \quad &\text{if } q\in (0,2\pi).
	\end{cases}
\end{align}

As for the bending energy, equation \eqref{EnergySmallWidth} now yields
\begin{equation*}
	\mathcal{E}(N(\theta_{q})) =\frac{w}{2}\int_{0}^{L} \left(\kappa\cos(\theta_{q}) \right)^{2}\left(1+\mu^{2}\right)^{2} \, dt.
\end{equation*}
Substituting \eqref{CaseBSolution}, we obtain
\begin{align}\label{CaseBEnergy}
\mathcal{E}(N(\theta_{q})) 	=
	\begin{dcases}
 \mathcal{E}(N(0))= \mathcal{E}(N)  \quad &\text{if }q=0,\\
\frac{w}{2}\int_{0}^{L} \frac{\left(1-\delta_{q}^{2}\right)^{2}}{\left(1+\delta_{q}^{2}\right)^{2}}\kappa^{2}\left(1+\mu^{2}\right)^{2} \, dt \quad &\text{if } q\in (0,2\pi),
	\end{dcases}
	\end{align}
where $\delta_{q}=\cot(q/2)+ \psi$.

\section{The helix}\label{Helix}
We conclude the paper by applying the results of the previous section to a specific curve, namely a circular helix of radius $a$ and pitch $2\pi b$:
\begin{equation*}
	\gamma(t)= \left( a\cos \frac{t}{\sqrt{a^{2}+b^{2}}}, a\sin\frac{t}{\sqrt{a^{2}+b^{2}}}, \frac{bt}{\sqrt{a^{2}+b^{2}}}\right).
\end{equation*}
The curvature and torsion are $a/(a^{2}+b^{2})$ and $b/(a^{2}+b^{2})$, respectively; they coincide with the normal curvature and geodesic torsion of $\gamma$ with respect to $\gamma''/\lVert \gamma''\rVert = N$.

We first examine the case $\alpha=\pi/2$. Setting $\varphi=\pi/2$ and $\tau_{g} = b/(a^{2}+b^{2})$, equation \eqref{newODE} becomes 
\begin{equation*}
\theta'  = - \frac{b}{a^{2}+b^{2}},
	\end{equation*}
and so
\begin{equation*}
	\theta(t) = - \frac{bt}{a^{2}+b^{2}}
\end{equation*}
is a solution, i.e., $\alpha(N(\theta)) = \pi/2$. It follows that the normal and geodesic curvatures with respect to $N(\theta)$ are
\begin{align*}
	\kappa_{n}(\theta(t)) &= \kappa\cos (\theta(t)) = \frac{a}{a^{2}+b^{2}} \cos\frac{bt}{a^{2}+b^{2}},\\
	\kappa_{g}(\theta(t)) &= \kappa\sin (\theta(t)) = -\frac{a}{a^{2}+b^{2}} \sin\frac{bt}{a^{2}+b^{2}},
	\end{align*}
whereas $\tau_{g}(\theta)=0$, as desired. Applying formula \eqref{EnergySmallWidthCase1} to $N(\theta)$, we obtain 
\begin{equation*}
	\mathcal{E}(N(\theta+q)) = a^{2} w \frac{2bL + \left(a^{2}+b^{2}\right)\left(\sin \left(\frac{2bL}{a^{2}+b^{2}}-2q\right)+\sin(2q)\right)}{8 b \left( a^{2}+b^{2} \right)^{2}}.
\end{equation*} 
Letting $r =bL/(a^{2}+b^{2})$ and normalizing by $\mathcal{E}(N(\theta))$, we finally get  
\begin{equation}\label{ratioA}
\frac{\mathcal{E}(N(\theta+q))}{\mathcal{E}(N(\theta))} = \frac{2r + \sin(2q) - \sin(2(q-r))}{2r +\sin(2r)}.
\end{equation} 

Figure \ref{figureA} displays the graph of the function $q\mapsto \mathcal{E}(N(\theta+q))/\mathcal{E}(N(\theta))$ for different values of the parameter $r$.

\begin{figure}[h]
	\centering
	
\subfloat[$r=1$]{\includegraphics[width=0.4\textwidth]{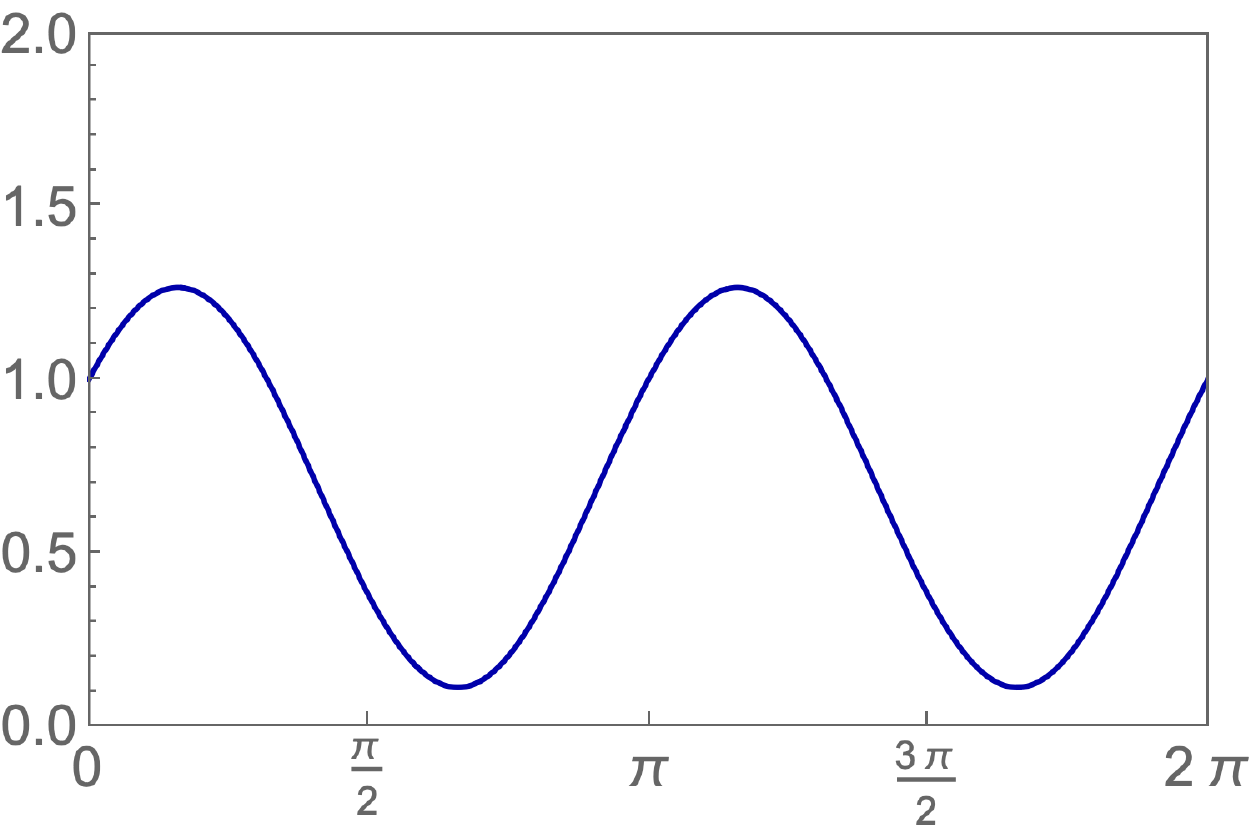}\label{sfA1}} \qquad
\subfloat[$r=2$]{\includegraphics[width=0.4\textwidth]{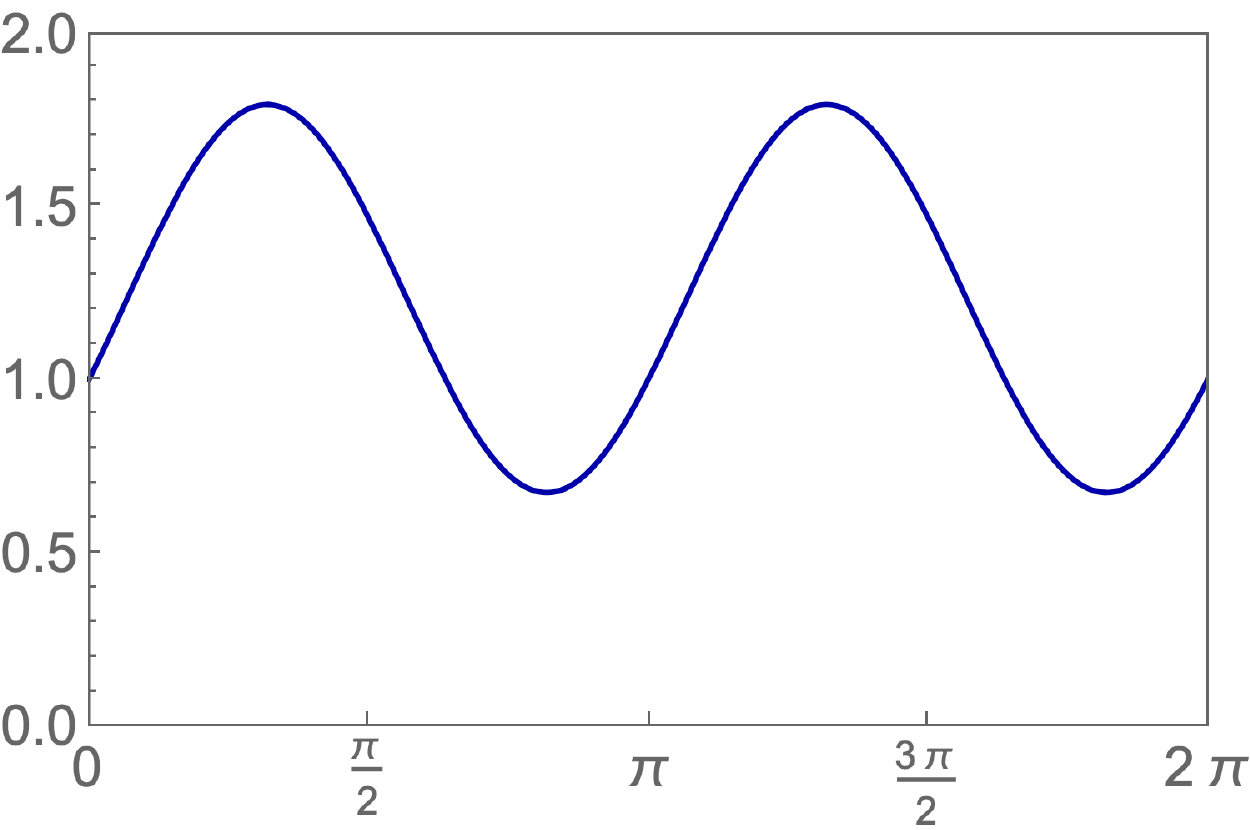}\label{sfA2}}

\subfloat[$r=3$]{\includegraphics[width=0.4\textwidth]{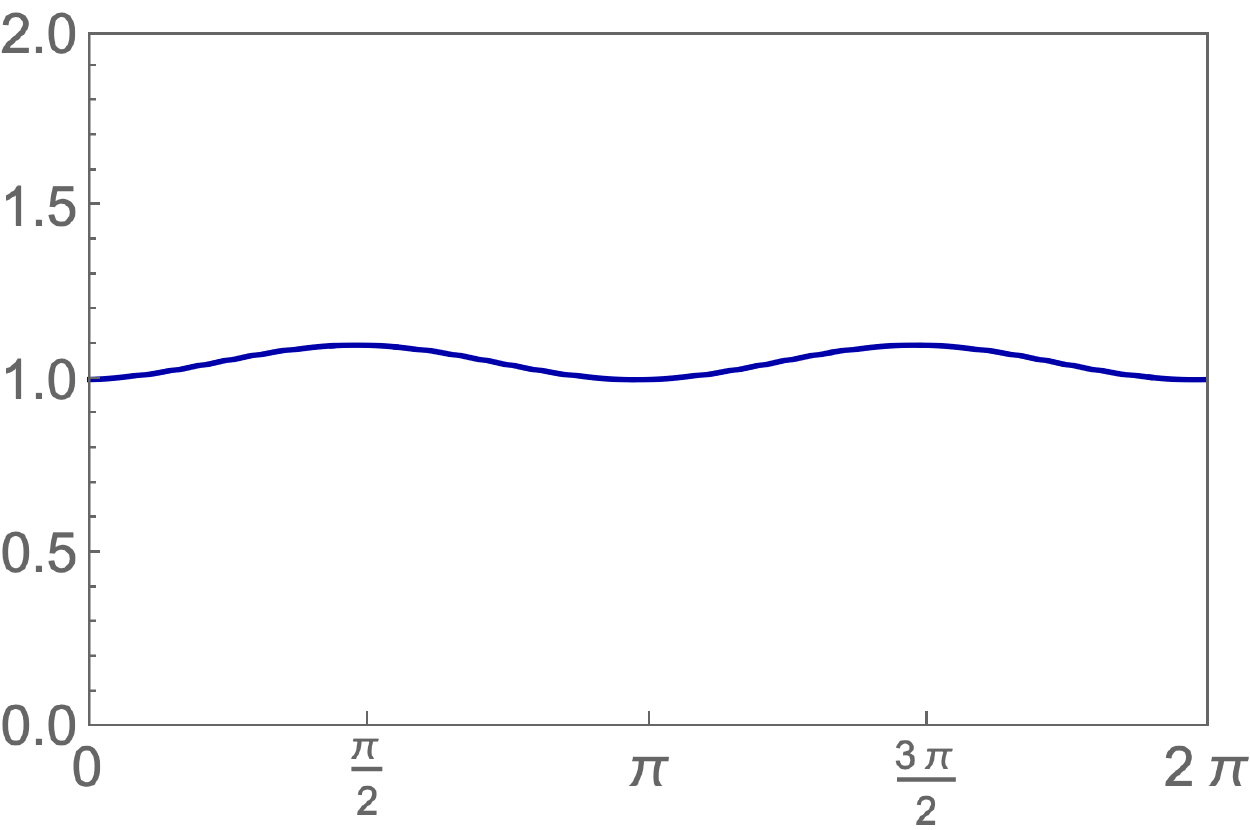}\label{sfA3}}\qquad
\subfloat[$r=4$]{\includegraphics[width=0.4\textwidth]{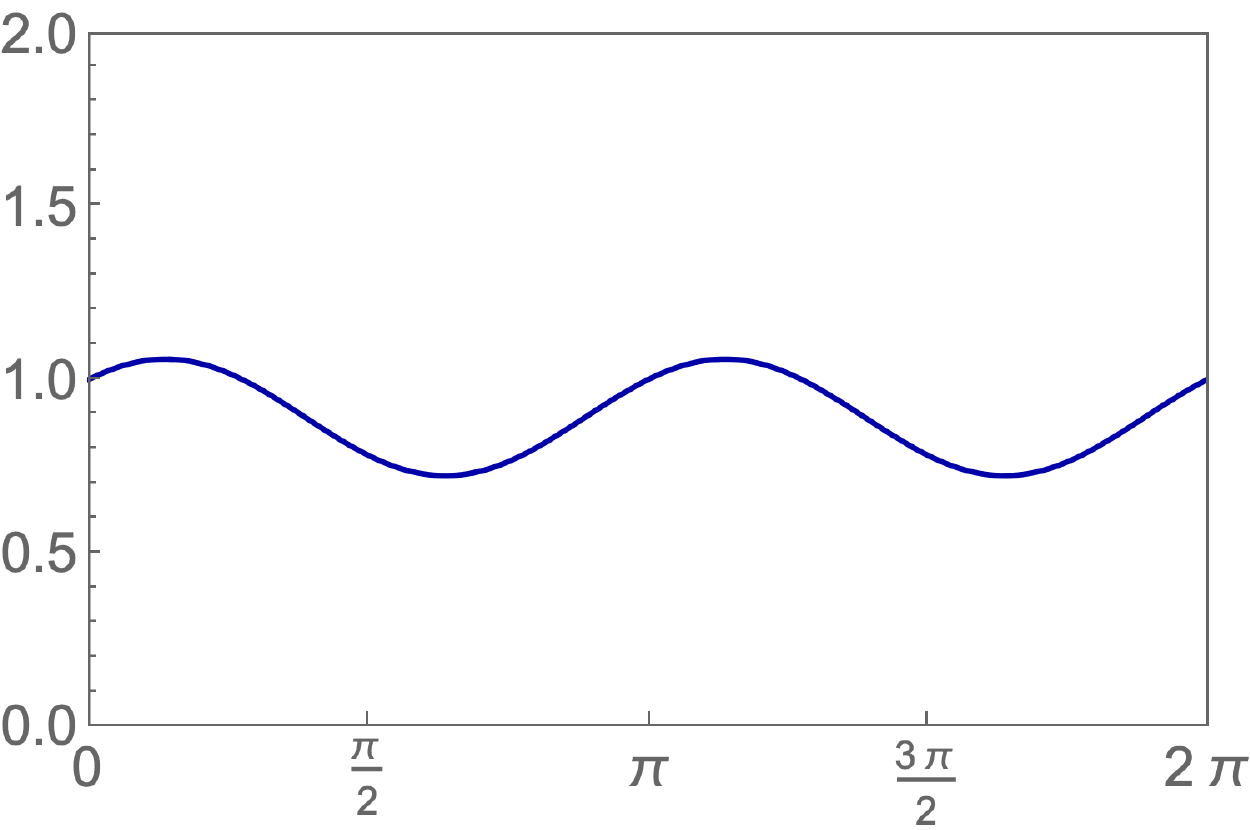}\label{sfA4}}
\caption[Plots of the normalized bending energy \eqref{ratioA}  as a function of $q$ for several values of $r$]{Plots of the normalized bending energy \eqref{ratioA} as a function of $q$ for several values of $r$.}\label{figureA}
\end{figure}

We then turn our attention to the case in which $\alpha$ coincides with the ruling angle of the rectifying developable. The bending energy is now given by \eqref{CaseBEnergy}. For $q=0$ it reads 
\begin{equation*}
 \mathcal{E}(N)  =\frac{wL}{2 a^{2}}.
\end{equation*}
On the other hand, if $q\neq 0$, then $\delta_{q}=\cot(q/2)+\psi$, where $\psi(t)= b t/(a^{2}+b^{2})$. A computation reveals that
\begin{equation*}
\frac{b}{a^{2}+b^{2}}\int \frac{\left(1-\delta_{q}^{2}\right)^{2}}{\left(1+\delta_{q}^{2}\right)^{2}}\, dt= -2\arctan \delta_{q} + \delta_{q} \frac{3+\delta_{q}^{2}}{1+\delta_{q}^{2}},
	\end{equation*}
from which one easily obtains $\mathcal{E}(N(\theta_{q}))$. In particular, it follows that
\begin{equation}\label{ratioB}
\frac{\mathcal{E}(N(\theta_{q}))}{\mathcal{E}(N)} =
\begin{multlined}[t]
\frac{1}{r} \Big[ 2\arctan(\cot(q/2)) -2 \arctan(\cot(q/2)+r)\\
{}+\left( \cot(q/2)+r\right) \frac{ 3 + \left( \cot(q/2)+r\right) ^{2}}{1 + \left( \cot(q/2)+r\right) ^{2}} \\
{}+\left(\cos(q) -2\right)\cot(q/2) \Big].
\end{multlined}
\end{equation}

As in the previous case, the graph of function $q\mapsto \mathcal{E}(N(\theta_{q})/\mathcal{E}(N)$ is plotted for different choices of $r$ in Figure \ref{figureB}.

\begin{figure}[h]
	\centering
	
	\subfloat[$r=1$]{\includegraphics[width=0.4\textwidth]{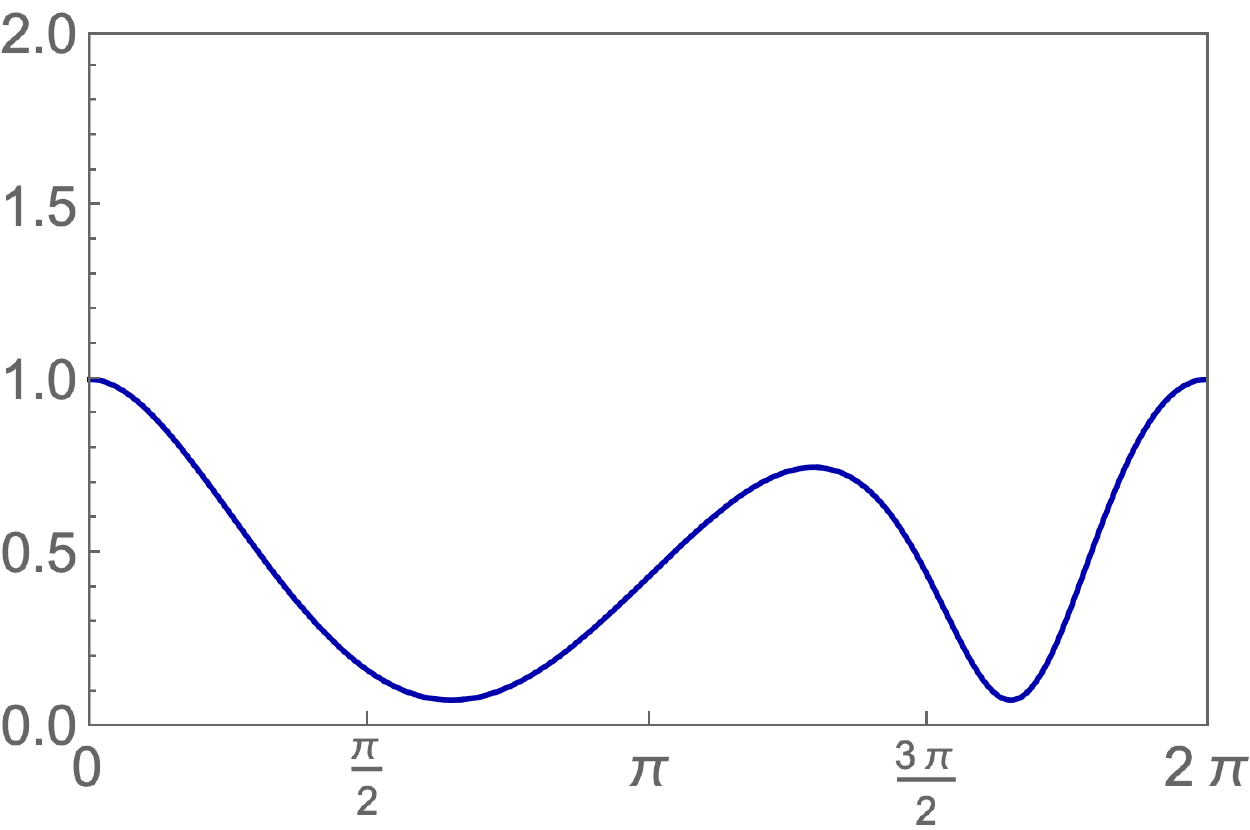}\label{sfB1}} \qquad
	\subfloat[$r=2$]{\includegraphics[width=0.4\textwidth]{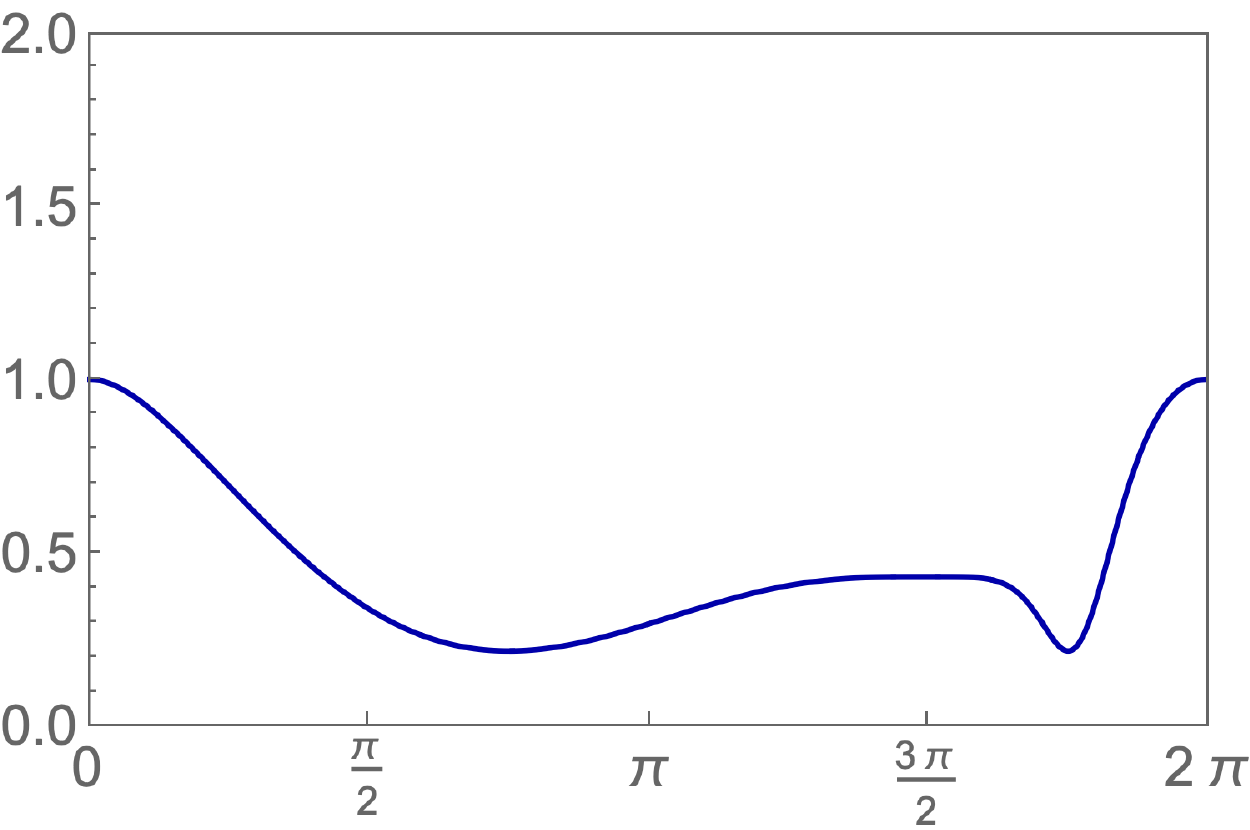}\label{sfB2}}
	
	\subfloat[$r=3$]{\includegraphics[width=0.4\textwidth]{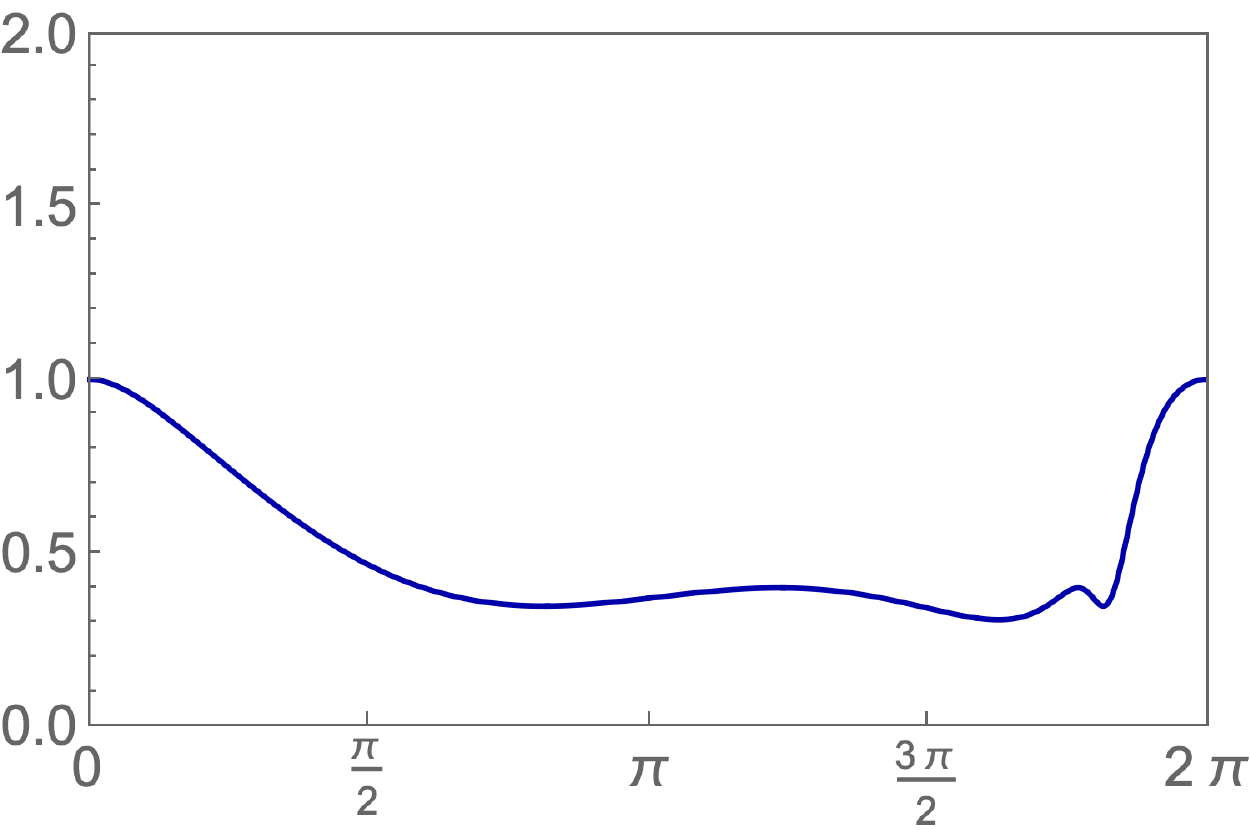}\label{sfB3}}\qquad
	\subfloat[$r=4$]{\includegraphics[width=0.4\textwidth]{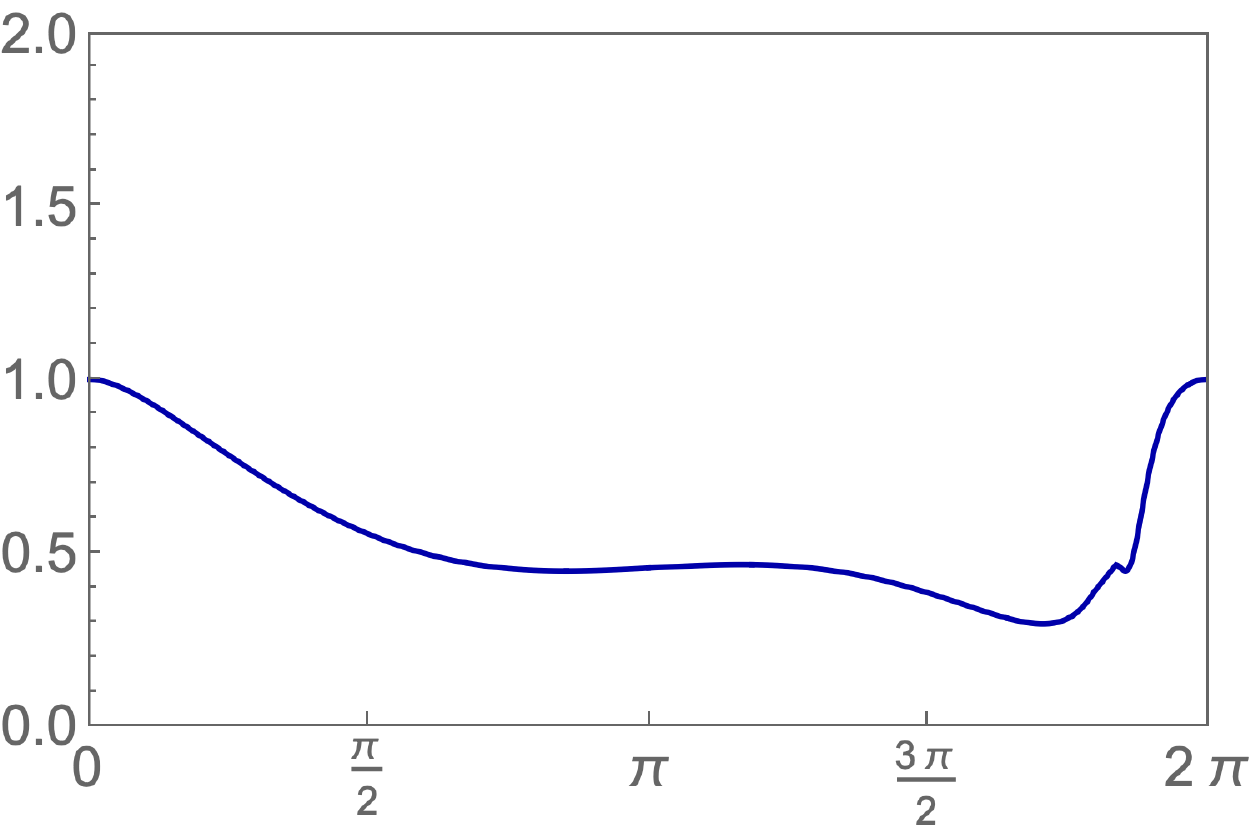}\label{sfB4}}
	\caption[Plots of the normalized bending energy \eqref{ratioB} as a function of $q$ for several values of $r$]{Plots of the normalized bending energy \eqref{ratioB} as a function of $q$ for several values of $r$.}\label{figureB}
\end{figure}

It is worth pointing out that, in both cases treated, the bending energy becomes less and less dependent on the initial condition as $r$ increases. More precisely, one can check that both ratios \eqref{ratioA} and \eqref{ratioB} tend to $1$ as $r \to \infty$.
%\begin{equation*}
%\lim_{r \to \infty} \frac{\mathcal{E}(N(\theta+q))}{\mathcal{E}(N(\theta))} = \lim_{r \to \infty} \frac{\mathcal{E}(N(\theta_{q}))}{\mathcal{E}(N)} =1.
%\end{equation*} 
It seems reasonable to expect that the same conclusion holds for any choice of ruling angle. Proving this is outside the scope of the present study.

\section*{Acknowledgments}
The author thanks David Brander, Christian M\"{u}ller, and an anonymous referee for many helpful comments and suggestions. 

\bibliographystyle{amsplain}
\bibliography{../biblio/biblio}
\end{document}